\documentclass{article}
\usepackage{graphicx} 
\usepackage[utf8]{inputenc}
\usepackage[T1]{fontenc}
\usepackage{amsthm,amsmath,amssymb,xcolor}
\usepackage{tikz,tkz-berge}
\usepackage{multicol}
\usepackage{enumerate}
\usetikzlibrary{arrows}
\usepackage{hyperref}
\usepackage{tabularx}
\usepackage{comment}
\usepackage[export]{adjustbox}
\usepackage{appendix}
\usepackage[normalem]{ulem}
\usepackage{authblk}
\usepackage{todonotes}	
\usepackage{comment}

\newtheorem{thm}{Theorem}

\newtheorem{cor}[thm]{Corollary}

\newtheorem{lemma}[thm]{Lemma}
\newtheorem{prop}[thm]{Proposition}
\theoremstyle{definition}

\newtheorem*{defn}{Definition}
\newtheorem{ex}{Example}
\newtheorem*{obs}{Observation}

\renewcommand\det{\operatorname{Det}}
\newcommand\dist{\operatorname{Dist}}
\newcommand\aut{\operatorname{Aut}}
\newcommand\stab{\operatorname{Stab}}

\usepackage{authblk}

\title{Twins and Co-Twins in Circulant Graphs}

\author{Sally Cockburn\thanks{Corresponding author: scockbur@hamilton.edu}, \enspace Ryhory Hatavets{\thanks{This research was supported by The Monica Odening '05 Student Internship and Research Fund in Mathematics.}} \enspace and Will Swartz}
\affil{Mathematics and Statistics Department\\ Hamilton College\\ Clinton, NY 13323, USA}

\date{\today}

\AtEndDocument{\bigskip{\footnotesize
 \medskip
\noindent   \textit{E-mail address}, (Sally Cockburn) \texttt{scockbur@hamilton.edu} 
  \addvspace{\medskipamount}
  
\noindent  \textit{E-mail address}, (Ryhory Hatavets) \texttt{Grishahatavets@gmail.com}

\noindent  \textit{E-mail address}, (Will Swartz) \texttt{willsadieswartz@gmail.com}
}}

\date{\today}

\begin{document}

\maketitle

\begin{abstract}    
    Circulant graphs are a widely studied family of graphs whose members  possess varying amounts of symmetry. Although considerable progress has been made in finding the automorphism groups of circulant graphs under certain restrictions, a complete classification is elusive.
    In general, the structure of the automorphism group of a graph with twins can be simplified by separating the effect of automorphisms that permute mutually twin vertices and those that operate on the twin quotient graph. 
    Further simplification can be achieved in twin-free, vertex-transitive graphs that have co-twins, which we
    define to be vertices whose neighborhoods are complementary. 
    In this paper, we demonstrate how the these simplifications can provide insight on, and in some cases completely determine, the automorphism groups and symmetry parameters of vertex-transitive graphs in general and circulant graphs in particular.      
\end{abstract}



 \section{Introduction}

 Let $n \in \mathbb N$ and let $A\subseteq   \mathbb Z_n \setminus \{0\}$ that is {\it inverse-closed} in the sense that $A = - A$. The circulant graph $C_n(A)$ has vertex set $\mathbb Z_n$, with $u,v \in \mathbb Z_n$ being adjacent iff $u-v \in A$. The elements of $A$ are called the generators of $C_n(A)$.  
 Each vertex in $C_n(A)$ has degree $|A|$; equivalently $C_n(A)$ is a $|A|$-valent (regular) graph.
 With their vertices drawn evenly spaced around a circle, circulant graphs certainly have what we would intuitively call a symmetric appearance; see the three examples in Figure~\ref{fig:TwinsCirculants}. 
 More technically, the symmetries of a graph $G = (V,E)$ are the bijections of $V$ that preserve adjacency and non-adjacency. Symmetries are also called graph automorphisms, and they form a group $\aut(G)$ under composition. The structure of $\aut(G)$ provides different ways of measuring of just how symmetric $G$ is.

 Any circulant graph $C_n(A)$ is vertex-transitive, because for any $u, v \in \mathbb Z_n$, $\alpha(w) = w-u+v$ is a graph automorphism satisfying $\alpha(u) = v$. However, not all circulant graphs are edge-transitive.
Note that $\tau_{-1}(a) = -a$ is  automorphism of $C_n(A)$, which implies that a circulant graph is edge-transitive iff it is arc-transitive.  The classification of all arc-transitive circulant graphs was found independently by Kovacs~\cite{K2004} and Li~\cite{L2005}. In ~\cite{PW2020}, Poto\u{c}nik and Wilson provide the implications of this classification for 4-valent circulant graphs; $3$-valent arc-transitive circulant graphs are given in ~\cite{GKLV2017}.

Another way to characterize the symmetry of a graph $G$ is to quantify how easy it is to block nontrivial automorphisms. For example, 
we could add the requirement that graph automorphisms fix each vertex in a set $D \subseteq V$; if the only automorphism that does so is the identity, then $D$ is called a determining (or fixing) set.
 The size of a smallest determining set, $\det(G)$, is called the determining number of $G$.  
 As another example, we could assign any one of $d$ colors to each vertex of $G$ and add the requirement that graph automorphisms must preserve (setwise) the color classes. If the only automorphism to do so is the identity, then the coloring is called $d$-distinguishing. The distinguishing number of a graph, $\dist(G)$, is the minimum number of colors required in a distinguishing coloring. In \cite{CL2024}, Cockburn and Loeb give the determining and distinguishing number of all $3$-valent and $4$-valent circulant graphs by first determining  their automorphism groups.

However, finding the automorphism group of circulant graphs in general is difficult. In 2007, Morris~\cite{M2007} provided an excellent survey of partial results, including when $C_n(A)$ is arc-transitive,  or when  $n$ is prime, a prime power, or square-free, and/or the elements of $A$ are all divisors of $n$. More recent results cover circulant graphs whose adjacency matrix has all integer eigenvalues or all rational eigenvalues; see \cite{BI2011} and \cite{KK2012} respectively. 

In this paper, we extend an approach used in \cite{CL2024} to investigate the automorphism group. Vertices 
are nonadajcent (respectively adjacent) twins if they have the same open (respectively closed) neighborhood.
In vertex-transitive graphs, the presence of twin vertices completely decides the determining number and reduces the task of computing the distinguishing number to computing that of its twin quotient graph.
In addition, applying a a classic theorem of Sabiduissi~\cite{S1964}, we can separate the effect of automorphisms that permute mutually twin vertices from those that those permute twin classes. 
In a few cases, iteratively looking at twin vertices in quotient graphs completely determines the automorphism group. In all cases, we can reduce the problem of finding automorphism groups of vertex-transitive graphs to finding automorphism groups of those that are twin-free.

Vertex-transitive graphs that are twin-free can still have co-twin vertices, defined in \cite{CL2024} to be vertices whose neighborhoods are complementary rather than equal. 
In such graphs, distinct sets of co-twin vertices cannot be independently permuted, and so their presence imposes strong constraints on the automorphism group.
If a vertex-transitive, twin-free graph with nonadjacent co-twins is also triangle-free, it must be a crown graph. If it has triangles, then its automorphism group, determining number and distinguishing number are related to those of the subgraph induced by the open neighborhood of any vertex.

The organization of the paper is as follows.
Section~\ref{sec:Circulants} reviews some basic facts about circulant graphs. 
In Section~\ref{sec:Twins}, we provide general results on the symmetry parameters and automorphism group of a vertex-transitive graph in terms of those of its twin quotient graph.
We apply these results to Cayley graphs in general and circulant graphs in particular in Section~\ref{sec:TwinsinCirculants}. As an application, we give a characterization of all connected, 5-valent and 6-valent circulant graphs with twins. 
In Section~\ref{sec:CoTwins}, we formally introduce co-twins and provide results on graphs that are vertex-transitive, twin-free and have co-twins.

\section{Circulant Graphs}\label{sec:Circulants} 

If $G$ is a group with inverse-closed subset $A \subseteq G \setminus \{e_G\}$, the corresponding Cayley graph $X(G,A)$ has vertex set $G$ with $g,h\in G$ being adjacent iff $gh^{-1} \in A$. For each $g \in G$, the function $\alpha_g(x) = xg$ is a graph automorphism of $X(G,A)$. For any $g, h \in G$, $\alpha_{hg{-1}}(g) = h$, meaning that $X(G,A)$ is vertex-transitive.  Moreover, for $g \neq e_G$, $\alpha_g$ has no fixed points. Thus $\aut(X(G,A))$ contains $G$ as a subgroup that acts regularly on the vertices; see \cite{GR2001}.
Conversely, in \cite{S1964}, Sabidussi  showed that if $X$ is a connected vertex-transitive graph and $G$ is a subgroup of $\aut(X)$ that acts transitively on $V(X)$, then there exists inverse-closed $A \subseteq G\setminus\{e_G\}$ such that  $X$ is the homomorphic image of $X(G,A)$. Additionally, he showed that the automorphim group of any vertex-transitive of order greater than $2$ is nonabelian. 

The circulant graph $C_n(A)$ is a Cayley graph with $G = \mathbb Z_n$. Using additive notation, $u,v \in \mathbb Z_n$ are adjacent iff  $u-v \in A$. It is immediate that the neighborhood of $0$ is $N(0) = A$; more generally, $N(u) = u+A$ for all $u \in \mathbb Z_n$. 
 Two extreme cases are $C_n(\mathbb Z_n  \setminus \{0\}) = K_n$ and $C_n(\emptyset) = N_n$; at the intersection of these two, there is the trivial circulant graph $C_1(\emptyset) = K_1.$  For $n = 2m$,  $C_n(A) = K_{mm}$ iff $A = \{1, 3, 5, \dots, n-1\}$.


The complement of $C_n(A)$ is $C_n(\overline A)$, where $\overline A$ is the complement of $A$ in $\mathbb Z_n \setminus \{0\}$. 
As shown in~\cite{Brooks2021}, $C_n(A)$ is connected iff $\gcd(n,A) = 1$; more generally, $C_n(A)$ has $\gcd(n,A)$ components. From \cite{H2003}, a connected circulant graph $C_n(A)$ is bipartite iff  $n$ is even and every $a \in A$ is odd.

For any unit $t \in U(n)$, it is easy to show that $C_n(A) \cong C_n(t \cdot A)$. In \cite{A1981}, \'Ad\'am conjectured that any isomorphism between two circulant graphs must be of this form, but this was found to be false in general. Considerable work has been done to find restrictions under which \'Ad\'am's conjecture holds; see \cite{M1997}. 



 If $u \in U(n)$ satisfies $u\cdot A = A$, then $\beta_u: \mathbb Z_n \to \mathbb Z_n$ given by $\beta_u(v)= uv$ is an automorphism of $C_n(A)$. We denote the set of all such automorphisms $\stab_n(A)$; it is easy to check that is it a subgroup of $\aut(C_n(A))$. 
As a special case of a result proved by Godsil  in \cite{G1981}, if $\mathbb Z_n$ is a normal subgroup of $\aut(C_n(A))$, then $\aut(C_n(A)) = \mathbb Z_n \rtimes \stab_n(A)$.

  \section{Twin Vertices}\label{sec:Twins}

  \begin{defn}
    The {\it open neighborhood} of a vertex $u$ in the graph $G = (V,E)$ is $N(u) = \{v \in V: uv \in E\}$ and the {\it closed neighborhood} of $u$ is $N[u] = N(u) \cup \{u\}$. Distinct vertices $u,v$ are
{\it nonadjacent twins} iff $N(u) = N(v)$
and  {\it adjacent twins} iff $N[u] = N[v]$.
    \end{defn}

    Nonadjacent twins are sometimes called false twins, with adjacent twins being true twins.

 \begin{lemma}\label{lem:adjNonadjTwins}
 \begin{enumerate}[(a)]
 \item A vertex cannot have both a nonadjacent and an adjacent twin.
\item Two vertices in a graph $G$ are nonadjacent twins in $G$ iff they are adjacent twins in the complement $\overline G$.
\end{enumerate}
\end{lemma}
    
\begin{proof}
For (a), assume $v$ has both nonadjacent twin $u$ and adjacent twin $w$.
Then $w \in N(v) = N(u)$ and so $u \in N[w] \setminus N[v] = \emptyset$, a contradiction. Statement (b) follows from 
 $\overline{N_G(v)} = V\setminus N(v)= N_{\overline G}(v) \cup \{v\} = N_{\overline G}[v]$. 
\end{proof}

In a vertex-transitive graph, either each vertex has a  nonadjacent twin, each vertex has a adjacent twin, or no vertex has either a nonadjacent or an adjacent twin.
In the final case, we say that $G$ is {\it twin-free}. (Sabidussi called such graphs {\it irreducible}.) Figure~\ref{fig:TwinsCirculants} gives examples of each of these within the family of circulant graphs. In the first two examples, twin vertices have the same color.

\begin{figure}[h]
    \centering
    \includegraphics[angle = 270, width= 0.85\textwidth, center]{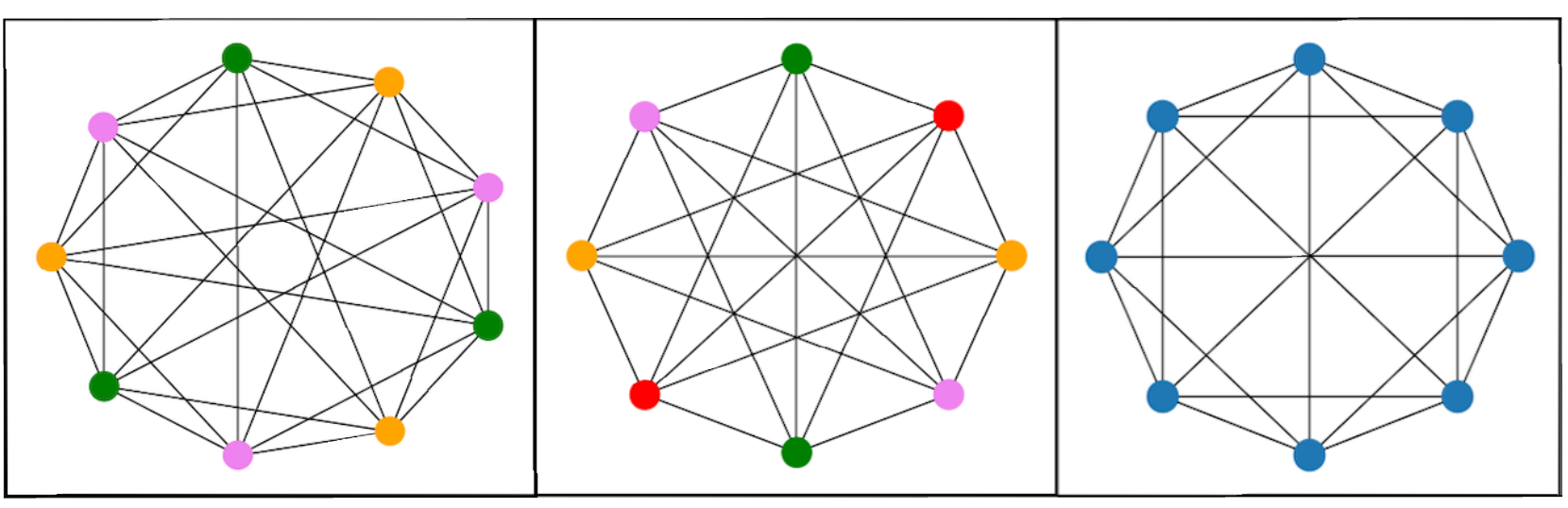}
    \caption{$C_9(\pm 1,\pm 2,\pm 4)$ (with nonadjacent twins), $C_{8}(\pm 1,\pm 3,4)$ (with adjacent twins) and $C_8(\pm 1,\pm 2,4)$ (twin-free).}
    \label{fig:TwinsCirculants}
\end{figure}

Using terminology from~\cite{BCKLPR2020b}, a {\it minimum  twin cover} $T$ of a graph $G$ is a subset of vertices that contains all but one vertex from each set of mutually nonadjacent twin vertices. The following is a corollary of Theorem~19 in~\cite{BCKLPR2020b}.

\begin{prop}\label{cor:BCKLPR}
If every vertex of $G$ has at least one nonadjacent twin, then a minimum twin cover of $G$ is a minimum size determining set for $G$. 
\end{prop}

Note that if every vertex of $G$ has an adjacent twin, then every vertex of $\overline G$ has a nonadjacent twin by Lemma~\ref{lem:adjNonadjTwins}. Any set of mutually adjacent twin vertices in $G$ will be a set of mutually nonadjacent twin vertices in $\overline G$, so a minimum twin cover in $\overline G$ will also be a minimum adjacent twin cover in $G$. A minimum determining set for $\overline G$ will also be a minimum determining set of $G$. 

\begin{cor}\label{cor:DetTwins}
Let $G$ be a vertex-transitive graph of order $n$  such that every vertex is in a set of  $t>1$ mutually (nonadjacent or adjacent) twin vertices. Then $\det(G) = n(1-\frac{1}{t})$.
\end{cor}

\subsection{Twin Quotient Graph}\label{subsec:TwinQuoGraph}

\medskip
 Define $u \sim v$ iff $N(u) = N(v)$  in a graph $G$. This is an equivalence relation  on $V(G)$; the corresponding quotient graph $\widetilde G$ is obtained by collapsing the vertices in each equivalence class to a single vertex. More precisely, the vertices of $\widetilde G$ are the equivalence classes 
 $[u] = \{v \in V(G) \mid u \sim v\}$, with $[u]$ and $[z]$ adjacent in $\widetilde G$ iff there exist $v \in [u]$, $w \in [z]$ such that $v$ and $w$ are adjacent in $G$. 
 By definition of $\sim$, all vertices in an equivalence class have the same neighbors, so in this case, distinct vertices $u$ and $z$ are adjacent in $G$ iff  $[u]$ and $[z]$ are adjacent in $\widetilde G$.
 We refer to $\widetilde G$ as the \emph{nonadjacent twin quotient graph} 
 and $[u]$ as the \emph{nonadjacent twin class} of $u$.
 Although $\widetilde G$  has no nonadjacent twins (see ~\cite{S1964}), it may have adjacent twins; for example, if $G = K_{2,3}$ then $\widetilde G = K_2$.

  Similarly, we can define $u \wedge v$ iff 
  $u$ and $v$ are adjacent twin vertices in $G$, and create a corresponding \emph{adjacent twin quotient graph}, which we denote by  $\widehat G$. In this case, distinct vertices $u$ and $z$ are adjacent in $G$ iff either $[u]=[z]$ or $[u]$ and $[z]$ are adjacent in $\widehat G$. 
  It is straightforward to verify that the adjacent twin quotient graph of the complement $\overline G$ is the complement of the nonadjacent twin quotient graph of $G$; it has no adjacent twins, but may have nonadjacent twins.

 \medskip

 \begin{ex}\label{ex:TumblingDown}
   Figure~\ref{fig:TumblingDown} displays a sequence of graphs obtained by alternatively collapsing adjacent and nonadjacent twins. 
 \end{ex}

 \begin{figure}[h]
    \centering
    \includegraphics[width= 0.75\textwidth, center]{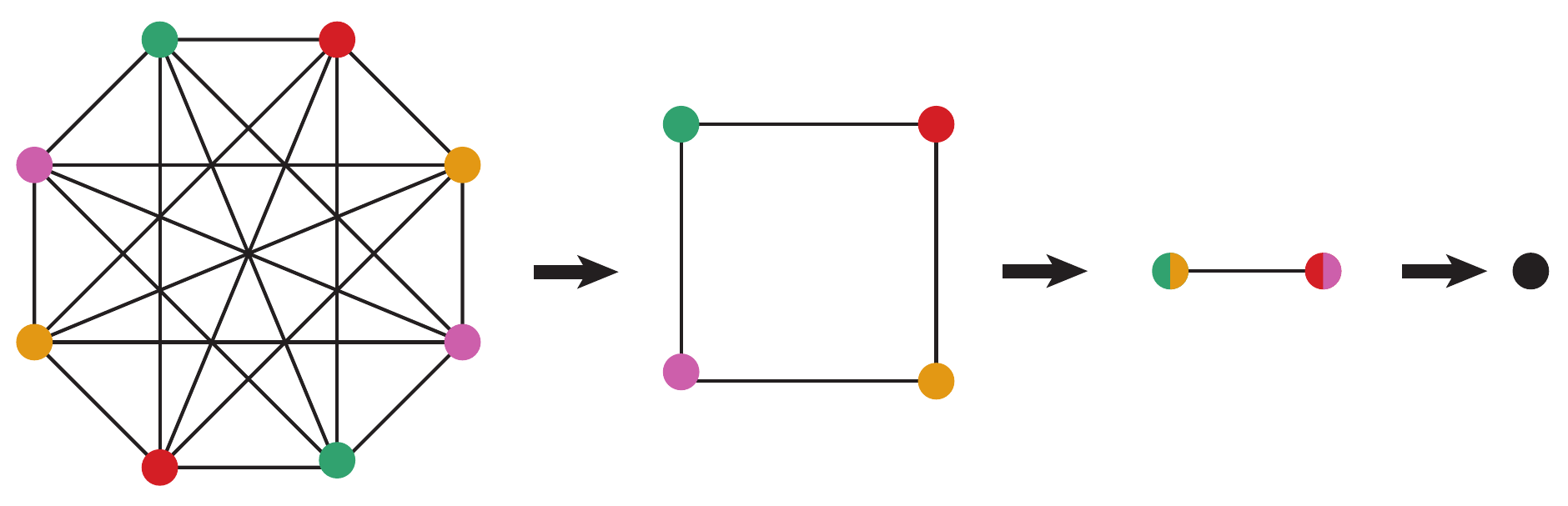}
    \caption{Twin quotient graph sequence of $C_8(\pm 1,\pm 3,4)$.}
    \label{fig:TumblingDown}
\end{figure}

The following result appeared first in \cite{BC2021} and is proved in \cite{CL2024}. 

\begin{thm}\label{thm:DistTwins}
Let $G$ be a graph in which every vertex is in a set of  $t>1$ mutually nonadjacent twin vertices. 
If $\dist(\widetilde G) = \widetilde d$, then $\dist(G) = d$, where $d$ is the smallest positive integer such that $\binom{d}{t} \ge \widetilde d$.
\end{thm}

Since $\dist(G) = \dist(\overline G)$, the same result holds for graphs in which every vertex is in a set of $t$ mutually adjacent twin vertices, and $\widetilde G $ is replaced with $\widehat G$.

\begin{ex}
Applying Theorem~\ref{thm:DistTwins} to the previous example, we start with the obvious fact that $\dist(K_1)=1$ to get $\dist(K_2) = 2$, $\dist(C_4) = 3$ and $\dist(C_8(\pm 1, \pm 3 , 4)) = 3$.
\end{ex}

The relationship between a vertex transitive graph and its twin quotient graph can be neatly described in terms of a graph product. The lexicographic product of two graphs $X$ and $Y$, denoted by $X[Y]$, has vertex set $V(X) \times V(Y)$, with $(x, y)$ adjacent to $(x', y')$ iff either $(x, x') \in E(\Gamma)$ or $x=x'$ and $(y,y') \in E(Y)$. If $G$ is a vertex-transitive graph with nonadjacent twins, then each twin class induces the null graph $N_t$ and by definition, $G \cong \widetilde G[N_t]$. 

A famous theorem by Sabidussi \cite{S1959} states that if $X$ and $Y$  are finite graphs, then $\aut(X[Y])$ is the wreath product $\aut(X) \wr \aut(Y)$ iff $Y$ is connected whenever $X$ has (distinct) nonadjacent twins and $\overline{Y}$ is connected whenever $X$ has (distinct) adjacent twins.
As noted earlier, $\widetilde G$ has no nonadjacent twins but it may have adjacent twins. Since $\overline N_t = K_t$ is connected. we can apply Sabidussi's theorem to obtain the following. 

\begin{prop}\label{prop:TwinsAutoQuotientSab}
  Let $G$ be a vertex-transitive graph in which every vertex is in  a set of $t>1$ mutually nonadjacent twin vertices. Then 
  \[
  \aut(G) = \aut(\widetilde G) \wr S_t.
  \]
\end{prop}

If $G$ has adjacent twins, we can apply this proposition to its complement.

\medskip

For those less familiar with the wreath product construction, we provide another formulation. Let $G$ be a vertex-transitive graph of order $n$, in which each vertex has a nonadjacent twin class of order $t$.
The function  $\tau_{u,v}:V(G) \to V(G)$ that interchanges distinct nonadjacent twins $u$ and $v$ and leaves all other vertices fixed is a graph automorphism, which we call a {\it twin transposition}.
The subgroup generated by all twin transpositions is $(S_t)^{n/t}$, a direct product of symmetric groups, one corresponding to each twin class.
Because graph automorphisms must take twin vertices to twins vertices, any $\alpha \in \aut(G)$ induces 
$\widetilde \alpha \in \aut(\widetilde G)$
given by $\widetilde \alpha([u]) = [\alpha(u)]$ 
for all $u \in V(G)$. Let $\lambda: \aut(G) \to \aut(\widetilde G)$ be the group homomorphism defined by $\lambda (\alpha) = \widetilde \alpha$. It is clear that $\ker(\lambda) = (S_t)^{n/t}$.

\begin{prop}\label{prop:TwinsAutoQuotient}
Let $G$ be a vertex-transitive graph with nonadjacent  twins. Then 
$\lambda$ is is surjective and $\aut(G) \cong (S_t^{n/t}) \rtimes \aut(\widetilde G)$.
\end{prop}

\begin{proof}
We can arbitrarily label the elements of a twin class $u_1, u_2, \dots, u_t$.
Given $\widetilde \sigma \in \aut(\widetilde G)$, define a permutation $\sigma $ on $V(G)$ by $\sigma(u_i) = v_i$ iff $\widetilde \sigma([u_i]) = [v_i]$. To show that $\sigma$ is an automorphism, let $x_j, y_k \in V(G)$. By definition of the twin quotient graph,  
\begin{align*}
x_j y_k \in E(G) & \iff [x_j][y_k] \in E(\widetilde G) \\
&\iff \widetilde \sigma([x_j])\widetilde \sigma([y_k]) \in E(\widetilde G) \iff \sigma(x_j)\sigma(y_k) \in E(G).
\end{align*}
Let $H = \{\gamma \in \aut(G) \mid \gamma(u_i) = v_i\}$; that is, $H$ is the set of automorphisms that respect the arbitrarily chosen subscripts. 
Then $H$ is a subgroup of $\aut(G)$ such that the restriction of $\lambda$ to $H$ is an isomorphism $\lambda: H \to \aut(\widetilde G)$. 
Note that $H$ is not a normal subgroup of $\aut(G)$; for example, if $\sigma(u_i) = v_i$, then  $\sigma\circ \tau_{u_1, u_2} (u_1) = v_2$ but $\tau_{u_1, u_2} \circ \sigma(u_1) = v_1$. Thus $\aut(G)$ is a semidirect product, not a direct product.
\end{proof}

\begin{ex} 
Given that Proposition~\ref{prop:TwinsAutoQuotient} applies to vertex-transitive graphs with either nonadjacent or adjacent twins, we get
\[ 
\aut(C_8(\pm 1, \pm 3, 4)) \cong (S_2)^3 \rtimes \aut(C_4) 
\cong (S_2)^3 \rtimes ((S_2)^2 \rtimes S_2).
\]
Here, the presence of twins completely determines the automorphism group.
\end{ex}

We close this section by noting that the property of being arc-transitive is preserved when passing to the twin quotient graph.

\begin{cor}\label{cor:TwinsArcTrans}
 Let $G$ be a vertex-transitive graph with nonadjacent twins. Then $G$ is arc-transitive iff $\widetilde G$ is arc-transitive. 
\end{cor}

\begin{proof}
Let $u, v, x$ and $y$ be vertices in $G$. 
First assume $G$ is arc-transitive and $[u][v], [x][y] \in E(\widetilde G)$. Then $uv, xy \in E(G)$.
By assumption, there exists $\alpha\in \aut(G)$ such that $\alpha(u) = x$ and $\alpha(v) = y$. Then 
$\alpha^*([u]) = [\alpha(u)] = [x]$ and $\alpha^*([v]) = [\alpha(u)] = [x]$.

Next, assume $\widetilde G$ is arc-transitive and  $uv, xy \in E(G)$. Then $[u][v], [x][y] \in E(\widetilde G)$.
By assumption, there exists $\widetilde \tau \in \aut(\widetilde G)$ such that $\widetilde \tau([u]) = [x]$ and $\widetilde \tau([v]) = [y]$. 
As in the proof of Proposition~\ref{prop:TwinsAutoQuotient}, we  arbitrarily select a representative $w_1 \in V(G)$ for each twin class  and label the elements of the class $w_1, w_2, \dots, w_t$. 
We define $\tau \in \aut(G)$ by $\tau(w_i) = z_i$ iff $\widetilde \tau([w]) = [z]$. 
In particular, $\tau(u_i) = x_j$ and $\tau(v_k) = y_\ell$ for some $i, j, k, \ell \in \{1, \dots, t\}$. We can compose $\tau$ with appropriate twin transpositions to find an automorphism of $G$ that takes $u$ to $x$ and $v$ to $y$.
\end{proof}

\section{Twins in Circulant Graphs}\label{sec:TwinsinCirculants}

We begin by characterizing  the set of nonadjacent twins in a general circulant graph $X(G,A)$. Recall that for each $g \in G$, $\alpha_g(x) = xg$ is an automorphism of $X(G,A)$. The open neighborhood of $g$ in $X(G,A)$ can be expressed as
\[
N(g) = \{x \in G \mid xg^{-1} \in A\} = \{x \in G \mid x= \alpha_g(a) \text{ for some } a \in A\} = \alpha_g(A).
\]

 By definition, $g$ is a nonadjacent twin of $e_G$ iff $N(g) = N(e_G)$ iff $\alpha_g(A) = A$. Letting $\stab(A)$ denote the stabilizer subgroup of $A$ in $\aut(X(G,A))$, we get that 
 the set $T$ of nonadjacent twins of $e_G$ in $X(G,A)$ is the subgroup
    \[
    T = G \cap \stab(A) = \{g \in G \mid \alpha_g(A) = A\} \le G.
    \]
\begin{lemma}
    Each right coset of $T$ in $G$ is a nonadjacent twin class of $X(G,A)$. Moreover, $A$ is a union of nontrivial right cosets of $T$.
\end{lemma}

\begin{proof} 
Let $x, y\in G$. By definition $y \in [x]$ iff $N(y) = N(x)$. As noted above, this can be written as $\alpha_y(A) = \alpha_x(A)$, or $Ay=Ax.$ This can be reformulated as $A(yx^{-1}) = A$ or $A = \alpha_{yx^{-1}}(A)$. By construction, this is equivalent to  $yx^{-1}  \in T$ or $y \in Tx$.

Let $a \in A$ and $t\in T$. Since $T \subseteq \stab(A)$,
$
A = \alpha_t(A) = \{at \mid a \in A\}$
 Since $at \in A$ for all $t \in T$, $aT \subseteq A$. Because $e_G \notin A$, each coset of $T$ in $A$ is nontrivial.        
 \end{proof}

 We now apply this to circulant graphs, where $G$ is the abelian group $\mathbb Z_n$, all of whose subgroups are cyclic. Thus $T = \langle w \rangle$ for some $w \in G$. If $w = 0$, then $C_n(A)$ has no (distinct) nonadjacent twins. 
If $w \in U(n)$, then $T = G$ and every nonzero vertex is a nonadjacent twin of $0$, implying that $C_n(A)$ is an empty graph (equivalently, $A = \emptyset$).

\begin{prop}\label{prop:WhenNonadjTwins}
    A nonempty circulant graph $C_n(A)$  has nonadjacent twins iff  there exists $0 \neq w \in \mathbb Z_n\setminus U(n)$ such that $A$ is a
    union of nontrivial cosets of $\langle w \rangle$. 
    In this case, the nonadjacent twin classes are the cosets of $\langle w \rangle$ where $w$ is the element of $\mathbb Z_n$ of maximum order satisfying the preceding statement.
\end{prop}

\begin{ex}
    For $C_{60}(\pm 1, \pm 9, \pm 11, \pm 19, \pm 21, \pm 29)$, 
    there are two ways to write $A$ as 
    a union of nontrivial cosets:
    \[
        A  = (1 + \langle 20 \rangle) \cup  (9 + \langle 20 \rangle) \cup  (11 + \langle 20 \rangle) \cup (19 + \langle 20 \rangle )\\
         = (1 + \langle 10 \rangle) \cup (9 + \langle 10 \rangle). 
    \]
    The generator of maximum order in this case is $w = 10$ and so by Proposition~\ref{prop:WhenNonadjTwins} the nonadjacent twin classes are the cosets of $\langle  10 \rangle$, each of size $6$.
\end{ex}

\begin{ex} 
We can use  Proposition~\ref{prop:WhenNonadjTwins} to find all circulant graphs of order $30$ in which the nonadjacent twin classes are the cosets of  $\langle 6 \rangle$. The orbits not containing $0$ are
$\mathcal O_1 = \pm 1 + \langle 6 \rangle$, 
$\mathcal O_2 =\pm 2 + \langle 6 \rangle$ and
$\mathcal O_3 = 3 + \langle 6 \rangle$. 
There are $2^3 - 1 = 7$ ways of constructing $A$; it is easy to verify that these generate  $7$ non-isomorphic circulant graphs (based on valency, connectedness and bipartiteness).
\end{ex}

Recall that $C_n(A)$ has adjacent twins iff $C_n(\overline A)$ has nonadjacent twins, where $\overline A$ is the complement of $A$ in $\mathbb Z_n\setminus \{0\}$.  Clearly $\overline A \neq \emptyset$ is a union of nontrivial cosets of $\mathbb Z_n$ iff $A \cup \{0\} \neq \mathbb Z_n$ is a union of cosets including the trivial coset.

\begin{cor}\label{cor:WhenAdjTwins}
    A non-complete circulant graph $C_n(A)$  has adjacent twins iff  there exists $0 \neq w \in \mathbb Z_n\setminus U(n)$ such that $A\cup \{0\}$ is 
    a union of cosets of $\langle w \rangle$. 
    In this case, the adjacent twin classes are the cosets of $\langle w \rangle$ where $w$ is the element of $\mathbb Z$ of maximum additive order satisfying the preceding statement.
\end{cor}

\begin{ex}
The circulant graph $C_{12}(\pm 1, \pm 5, 6)$ has no nonadjacent twins  because  $\gcd(12, |A|) = 1$. However,  $A \cup \{0\} = \langle 6\rangle \cup (1 + \langle 6\rangle)  \cup (5 + \langle 6 \rangle)$ 
so $C_{12}(\pm 1, \pm 5, 6)$ has adjacent twins. 
    \end{ex}

A \emph{two-generator} circulant graph has a generator set of the form $A = \{\pm i, \pm j\}$ and can be denoted $C_n(i,j)$ where $0<i<j\le n/2$. If $j = n/2$, then 
$C_n(A)$ is $3$-valent; otherwise it is $4$-valent.
In \cite{CL2024}, Cockburn and Loeb characterize all connected, two-generator circulant graphs with twins.

\begin{lemma}[\cite{CL2024}]\label{lem:whenTwins}
Let $0 <i<j\le n/2$, with $\gcd(i,j,n) = 1$.
Then $C_n(i,j)$
has twins only in the cases shown in Table~\ref{tab:TwoGenTwins}. The twin classes in the non-complete graphs are the cosets of $\langle w \rangle$.
\begin{table}[h]
    \center
    \begin{tabular}{|c|c|c|} \hline
     & Nonadjacent twins & Adjacent twins\\
    \hline
    $j = n/2$ & $C_6(1,3): w=2$ & $C_4(1,2) = K_4$\\ \hline 
    $j < n/2$ & $C_8(1, 3): w=2 $  & $C_5(1,2) = K_5$ \\
     & $i+j=\tfrac{n}{2}: w= \tfrac{n}{2} $ & \\
     \hline   
    \end{tabular}
    \caption{Connected two-generator circulant graphs with twins.\label{tab:TwoGenTwins}}
\end{table}
\end{lemma}

The proof in \cite{CL2024} uses elementary modular arithmetic techniques. Proposition~\ref{prop:WhenNonadjTwins} and Corollary~\ref{cor:WhenAdjTwins} allow for a more streamlined proof of the following characterization of all connected, three-generator circulant graphs with twins.

\begin{prop}\label{prop:Twins3Gen}
Let $0<i<j<k\le n/2$, with $\gcd(i,j,k,n) = 1$.
Then $C_n(i, j, k)$  has twins only in the cases
shown in Table~\ref{tab:ThreeGenTwins}. The twin classes of the non-complete graphs are the cosets of $\langle w \rangle$.

\begin{table}[h]
    \center
    \begin{tabular}{|c|c|c|} \hline
     & Nonadjacent twins & Adjacent twins\\
    \hline
    $k = n/2$ & $C_{10}(1,3,5) = K_{5,5}: w = 2$ & $C_6(1,2,3) = K_6$\\
    &&$i+j=k = \tfrac{n}{2}: w = \tfrac{n}{2}$ \\ \hline 
     & $C_{12}(1, 3, 5) = K_{6,6}: w = 2$ & $C_7(1,2,3) = K_7$ \\
    $k < n/2$ &   $i+j=\tfrac{n}{3}, 2i+j=k: w = \tfrac{n}{3}$ & \\
    &  $i+k = 2j = \tfrac{n}{2}: w = \tfrac{n}{2}$ &\\ \hline  
    \end{tabular}
    \caption{Connected three-generator circulant graphs with twins.\label{tab:ThreeGenTwins}}    
\end{table}
\end{prop}

\begin{proof} First assume $k = n/2$ and that $C_n(i,  j , k)$ has nonadjacent twins. 
Since $|A|=5$ is prime, $A$ must be a single nontrivial coset of some $w \in \mathbb Z_n$ of order $5$. 
If we choose $n/5 \in \mathbb Z$ to represent the congruency class $w \in \mathbb Z_n$, then the two integer sequences  \[0<i< i+\tfrac{n}{5}< i+\tfrac{2n}{5}< i+\tfrac{3n}{5}< i+\tfrac{4n}{5} <n \]
and 
\[0<i<j<k=n/2 <n-j < n-i <n\]
must match up.
Simple algebra gives $j= 3i$, $k = 5i$ and $n = 10i$. For this circulant graph to be connected, $1 = \gcd(n, i, j, k ) = i$, so it must be $C_{10}(1,3,5)$.

Now assume $C_n(i,j,k)$ has adjacent twins. By Corollary~\ref{cor:WhenAdjTwins}, 
$A \cup\{0\}$ is a union of cosets contianing the trivial coset. Since $|A\cup \{0\}|=6$, there are three cases to consider. If $A\cup \{0\}$ is a single trivial coset of  some $w \in \mathbb Z_n$ of order $6$, then 
we can represent $w$ with $\frac{n}{6}$.
A matching argument similar to the one above gives $j = 2i$, $k=3i$ and $n = 6i$;
the only such circulant graph that is connected is $C_6(1,2,3) = K_6$. (In this case, $w$ is the unit $1 \in U(6)$.)
If $A \cup \{0\}$ is a union of two cosets of some $w$  of order $3$, we get the same result.
Finally, if $A\cup \{0\}$ is a union of three cosets of some $w \in \mathbb Z_n$  of order $2$, then we can represent $w$ with $k = n/2$. Then $A = \{0, k\} \cup \{i, i+k\} \cup \{j, j+k\}$; since $A$ is inverse-closed, $i+j = k$. 

The proof in the case $k < n/2$ is similar and appears in the appendix.
\end{proof}

In \cite{S1964}, Sabidussi shows that the nonadjacent twin quotient graph of a graph Cayley graph is also a Cayley graph, which has no (distinct) nonadjacent twins; more precisely,
\[
\widetilde{X(G,A)} = X(G/T, A/T).
\]


The proposition below formulates this result in the context of circulant graphs. For clarity, the elements of $\mathbb Z_k$ are denoted by $[u]_k$, where $u \in \mathbb Z$.

\begin{prop}\label{prop:QuotientCirculantisCirculant}
Let $C_n(A)$ be a circulant graph with nonadjacent twin classes of size $0<t$.
Let $m = n/t$. 
Then its nonadjacent twin quotient graph  
is $C_m (A')$, where $A'=\{[a]_m: [a]_n \in A\}$. 
\end{prop}

\begin{ex}
In $C_9(\pm 1,\pm 3, \pm 4)$, each nonadjacent twin class has size $t=3$. Thus $m = 3$ and  
$A'  = \{[1]_3, [2]_3\} = \{ \pm [1]_3\}$.
Hence the nonadjacent twin quotient graph of $C_9(\pm 1,\pm 3, \pm 4)$ is $C_3(\pm 1) = C_3$.
\end{ex}

In \cite{CL2024}, Cockburn and Loeb show that all connected 3-valent and 4-valent circulant graphs with nonadjacent twins are arc-transitive, based on the characterization of such graphs by Poto\u{c}nik and Wilson~\cite{PW2020}.
As an alternate proof, Proposition~\ref{prop:QuotientCirculantisCirculant} implies that the twin quotient graph of a 3-valent or 4-valent circulant graph is a 2-valent or 1-valent circulant graph. 
Since such graphs are clearly arc-transitive, the result follows from Corollary~\ref{cor:TwinsArcTrans}.  However, circulant graphs  of higher valency with nonadjacent twins may not be arc-transitive.  An example of a 6-valent circulant graph with twins that is not arc-transitive is $C_{12}(\pm 2, \pm 3, \pm 4)$; its twin quotient graph is $C_6(\pm 2, 3)$, which is not arc-transitive.

\section{Co-Twins}\label{sec:CoTwins}

As shown in Section~\ref{subsec:TwinQuoGraph}, by alternatively passing to nonadjacent and adjacent twin quotient graphs as necessary, any vertex-transitive graph $G$ has a twin-free graph $G_0$ as a homomorphic image.
By 
Theorem~\ref{thm:DistTwins}, the problem of finding $\dist(G)$ can be reduced to finding $\dist(G_0)$. 
More generally, Proposition~\ref{prop:TwinsAutoQuotient} can be used to express $\aut(G)$ as a semidirect product of permutation groups and $\aut(G_0)$. 
The task of finding the distinguishing number and automorphism group of vertex-transitive, twin-free graphs can in some cases be further simplified by investigating a similar vertex relationship.

\begin{defn} Distinct vertices $u$ and $v$ in graph $G = (V,E)$ are 
{\it nonadjacent co-twins} iff $N[u] = \overline{N[v]}$ (equivalently, $N[u] \sqcup N[v] = V$),
   and {\it adjacent co-twins} iff $N(u) = \overline{N(v)}$ (equivalently, $N(u) \sqcup N(v) = V$).
\end{defn}

\begin{ex}
    The hypercube and icosahedral graph are vertex-transitive, twin-free with nonadjacent co-twins;   the envelope graph is vertex-transitive, twin-free with adjacent co-twins.  See Figure~\ref{fig:EZCoTwinsIso}, in which co-twins share the same color.
\end{ex}

\begin{figure}[h]
    \centering
    \includegraphics[width= 0.90\textwidth, center]{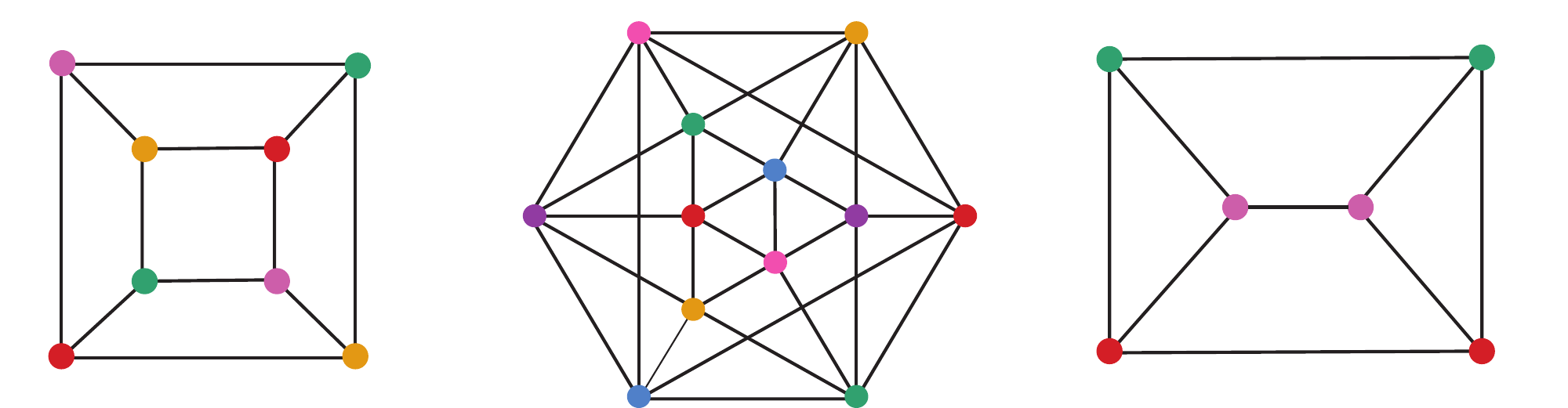}
    \caption{The hypercube, the icosahedral graph and the envelope graph.}
    \label{fig:EZCoTwinsIso}
\end{figure}


  \begin{lemma}\label{lem:adjNonadjCoTwins}
 \begin{enumerate}[(a)]
 \item A vertex cannot have both a nonadjacent and an adjacent co-twin.
\item Two vertices in a graph $G$ are nonadjacent co-twins in $G$ iff they are adjacent co-twins in the complement $\overline G$.
\end{enumerate}
\end{lemma}

\begin{proof} 
Suppose $v \in V$ has  nonadjacent co-twin $u$ and adjacent co-twin $w$.
Then $N[u] \sqcup N[v] = V = N(w) \sqcup N(v)$. 
Because $w$ is adjacent to $v$,  $w \in N(v)\subset N[v]$.
Because $u$ is not adjacent to $v$, $u \notin N(v)$ and so $u \in N(w)$. Since $u$ is adjacent to $w$, $w$ is adjacent to $u$ and so $w \in N [u]$. This is a contradiction as $N[u]$ and $N[v]$ are disjoint.

For (b), from the proof of Lemma~\ref{lem:adjNonadjTwins}, $\overline{N_G(v)} = N_{\overline G}[v]$. By complementarity, this implies both 
$\overline{N_{\overline G}(u)} = N_G[u] $ and $\overline{N_G[v]} = N_{\overline{G}}(v)$.
Thus $N_G[u] = \overline{N_G[v]}$ iff $N_{\overline{G}}(v) = \overline{N_{\overline G}(u)}$.
\end{proof}

Lemma~\ref{lem:adjNonadjCoTwins} shows that in some ways, co-twins are similar to twins. An important difference is that being (nonadjacent or adjacent) co-twins is not an equivalence relation because it is neither reflexive nor transitive. However, there is an analog of the twin class of a vertex.

\begin{lemma}\label{lem:AtMostOneCoTwin}
    If $G$ has no adjacent (resp. nonadjacent) twins, then any vertex in $G$ has at most one nonadjacent (resp. adjacent) co-twin. 
\end{lemma}

 \begin{proof} 
If $u$ and $w$ are both nonadjacent co-twins of $v$, then $N[u]= \overline{N[v]} = N[w]$. By definition, $u$ and $w$ are adjacent twins in $G$.  Similarly, if $u$ and $w$ are both adjacent co-twins of $v$, then  $N(v) = \overline{N(u)} = N(w)$ and so $v$ and $w$ are nonadjacent twins.
 \end{proof}

 If $G$ is vertex-transitive, twin-free, with nonadjacent co-twins, then $\overline G$ is vertex-transitive, twin-free with adjacent co-twins. Moreover $G$ and $\overline G$ have the same automorphism group and symmetry parameters. To simplify terminology, henceforth we focus only on nonadjacent co-twins,  calling them simply co-twins. 

\begin{cor}
 Let $G$ be a vertex-transitive, twin-free graph of order $n$ with 
 co-twins. Then $V(G)$ can be partitioned into pairs of 
 co-twins, implying that $n$ is even. Moreover, every vertex has degree $\frac{n}{2} - 1$.   
\end{cor}

\begin{ex}\label{ex:crown}
It is easy to verify that for all $k \ge 3$, $K_k \times K_2$ is vertex-transitive and twin-free. If we let $V(K_2) = \{0, 1\}$, then for any $u \in V(K_k)$, the vertices $(u,0)$ and $(u,1)$ are adjacent co-twins in $K_k \times K_2$. By Lemma~\ref{lem:adjNonadjCoTwins}, these vertices are nonadjacent co-twins in (vertex-transitive, twin-free) $\overline{K_k \times K_2}$.  Note that the envelope graph is $K_3 \times K_2$ and the hypercube is $\overline{K_4\times K_2}$. The graphs $\overline{K_k \times K_2}$ are sometimes called crown graphs. As Figure~\ref{fig:CrownGraphEg} illustrates, $\overline{K_k \times K_2}$ can also be described as $K_{k,k}$ minus a 1-factor.
\end{ex}

\begin{figure}[h]
    \centering
    \includegraphics[width= 0.35\textwidth, center]{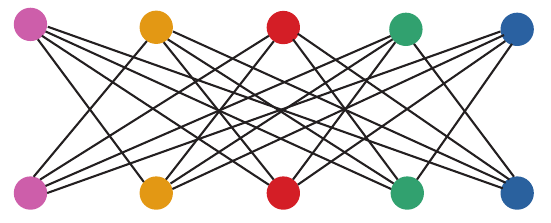}
    \caption{The crown graph $\overline{K_5 \times K_2}$.}
    \label{fig:CrownGraphEg}
\end{figure}

Assume the circulant graph $C_{2k}(A)$ is twin-free with 
co-twins.
By definition, the 
co-twin $v$  of $0$ satisfies\[
V = \mathbb Z_n = N[0] \sqcup N[v] = \{0\} \sqcup A \sqcup \{v\} \sqcup (v+A).
\]
 Since $A$ is inverse-closed, multiplying both sides of the equation above by $-1$ gives $\mathbb Z_n = N[0] \sqcup N[-v]$ and so 
$N[v] = N[-v]$. By the assumption that $C_n(A)$ is twin-free,
$v = -v$ in $\mathbb Z_{2k}$ and so $v = k$. Since translation by $u \in \mathbb Z_n$ is a graph automorphism, every pair of 
co-twins is of the form $\{u, u+k\}$.

\begin{lemma}\label{lem:NonadjCoTwins}
Twin-free $C_{2k}(A)$ has 
co-twins iff  $k \notin A$, $|A|= k-1$, and for all $a \in A$, $k+a\in \overline A$.
\end{lemma}

\begin{proof}
First assume that twin-free $C_{2k}(A)$ has 
co-twins.
As noted above, the 
co-twin of $0$ is $k$ and so
$
\mathbb Z_{2k} = N[0] \sqcup N[k] = \{0\} \sqcup A \sqcup\{k\} \sqcup (k+A).
$
Because all sets are disjoint, $k \notin A$ and $n = 2 + 2|A|$.  Hence $|A|= k-1$.
Moreover,  for all $a \in A$, $k+a \in k +A$ is in the complement of $A$ in $\mathbb Z_n \setminus\{0\}$, denoted $\overline A$ in this paper.


Conversely, assume $k \notin A$, $|A|= k-1$, and for all $a \in A$, $k+a\in \overline A$.
Combined with our usual assumptions that $0 \notin A$ and $A$ is inverse-closed, 
the sets $\{0\},  A, \{k\}$ and $k+A$ are disjoint.  
Simple enumeration shows that $|A|= k-1$ implies $\mathbb Z_{2k} = N[0] \sqcup N[k]$. By definition $C_{2k}(A)$ has 
co-twins.
\end{proof}

\begin{cor}\label{cor:OrderTwinFreeCoTwins}
    If $C_{2k}(A)$ is twin-free with 
    co-twins, then $k$ is odd.
\end{cor}

\begin{proof}
    Since $A= -1 \cdot A$, the unique additive inverse of each element of $A$ is also in $A$. By Lemma~\ref{lem:NonadjCoTwins}, $k\notin A$, so multiplication by $-1$ doesn't fix any element of $A$.  Thus  $|A| = k-1$ is even, forcing $k$ to be odd.
\end{proof}

\begin{ex}
    There are two non-isomorphic circulant graphs of order $ 10$ that are twin-free with 
    co-twins, namely $C_{10}(\pm 1, \pm 2)$ and $C_{10}(\pm 1, \pm 3)$. 
    (Cockburn and Loeb assert in \cite{CL2024} that $C_{10}(\pm 1, \pm 3)$ is the only two-generator circulant graph that is twin-free with co-twins; in fact, it is the only such graph that is also triangle-free.)
    There are three non-isomorphic circulant graphs of order $14$ that are  twin-free with 
    co-twins, namely  $C_{14}(\pm 1, \pm 2, \pm 3)$, $C_{14}(\pm 1, \pm 2, \pm 4)$ and  $C_{14}(\pm 1, \pm 3, \pm 5)$.     
\end{ex}

The following proposition shows the relationship between circulant graphs and the crown graphs of Example~\ref{ex:crown}

\begin{prop}\label{prop:WhenCirculantCrown}
    Let $3 \le k \in \mathbb N$. 
    If $k = 2\ell +1$, then $\overline{K_k \times K_2} \cong C_{2k}(A)$, where $A= \{\pm 1, \pm 3, \dots,  \pm (2\ell-1)\}$. If $k$ is even, then $\overline{K_k \times K_2}$ is not isomorphic to a circulant graph.
\end{prop}

\begin{proof}
    As noted in the introduction, $C_{2k}(A) = K_{k,k}$ iff $A = \{1, 3, 5, \dots, 2k-1\}$. If $k$ is odd, then it is  contained in this set. Removing $k$ from the generator set amounts to removing a 1-factor from $K_{k,k}$ and $\overline{K_k \times K_2}$ is the same thing as $K_{k,k}$ minus a 1-factor.

    Next assume $k$ is even and $\overline{K_k \times K_2} \cong C_{2k}(A')$ for some $A' \subset \mathbb Z_{2k}\setminus \{0\}$. Crown graphs are known to be Hamiltonian \cite{LA1976}; in order for $C_{2k}(A')$ to be Hamiltonian, we can assume WLOG that $1 \in A'$. Every vertex in $\overline{K_k \times K_2}$ has degree $k-1$ which is odd, implying that $|A'|$ must be odd. Since $A'$ is inverse-closed, $k \in A'$. Thus $C_{2k}(A')$ contains the odd cycle $(0, 1, 2, \dots, k, 0)$. However, $\overline{K_k \times K_2}$ is bipartite, a contradiction.
\end{proof}

\subsection{Co-Twin Quotient Graph}

Throughout this subsection, we assume $G$ is a graph of order $n = 2k$ that is vertex-transitive, twin-free, with  
co-twins. This implies $k \ge 3$. Although the relation of being 
co-twins is not an equivalence relation, we can still form a co-twin quotient graph  by collapsing each 
co-twin pair $\{u, v\}$ into a single vertex denoted by $[u,v]$ and adding an edge between collapsed pairs $[u,v]$ and $[x,y]$ iff a vertex in $\{u,v\}$ is adjacent in $G$ to a vertex in $\{x,y\}$.

\begin{lemma}\label{lem:K2UK2} 
 If $\{u,v\}$ and $\{x,y\}$ are distinct pairs of 
 co-twins, then the induced subgraph $G[u,v,x,y]$ is $K_2 \cup K_2$.   
\end{lemma}

\begin{proof}
Since $x \in V(G) = N[u] \sqcup N[v]$, we can assume without loss of generality that $x \in N[u]$. Because these are distinct pairs, $x \in N(u)$.  Since $N[x]$ and $N[y]$ are disjoint, $y$ cannot also be adjacent to $u$ and so $y \in N(v)$.
\end{proof}

\begin{figure}[h]
    \centering
    \includegraphics[width= 0.32\textwidth, center]{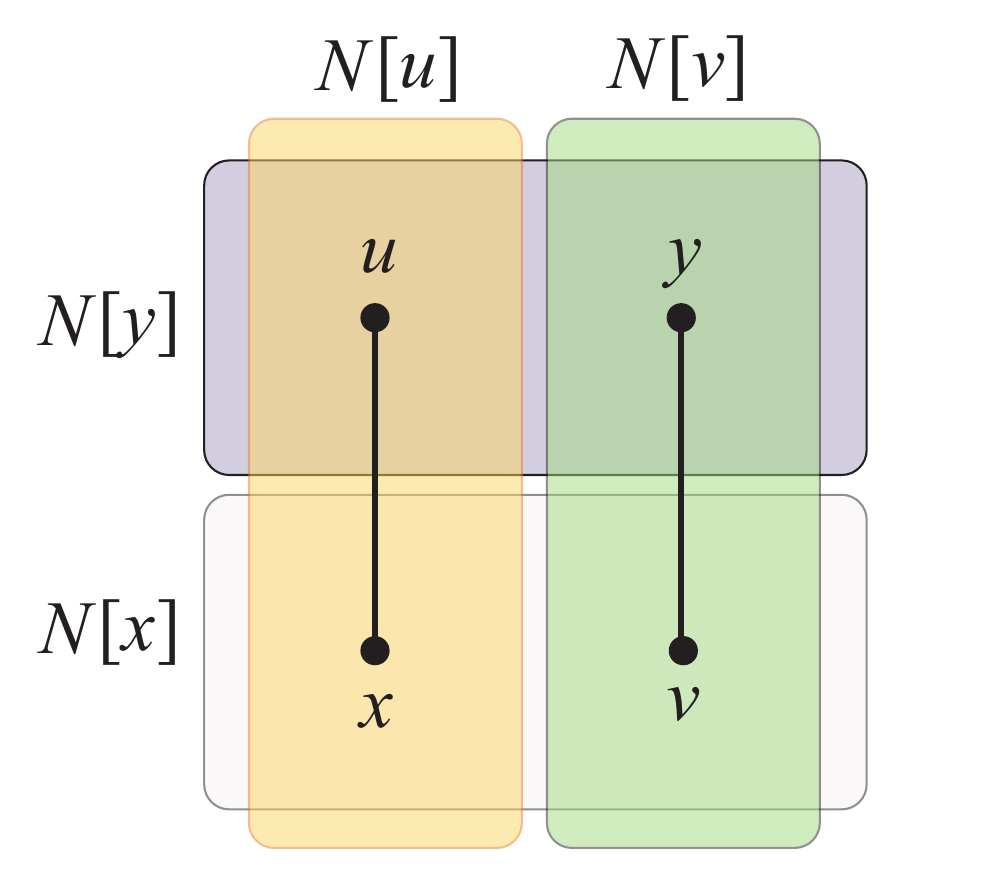}
    \caption{Subgraph induced by 
    co-twin pairs $\{u,v\}$ and $\{x,y\}$.}
    \label{fig:K2UK2}
\end{figure}

By this lemma, the co-twin quotient graph is simply $K_k$. 
Thus the function $G \to K_k$ that collapses each 
co-twin pair into a single vertex is by definition a surjective graph homomorphism.

As with twin vertices, this leads to a relationship between the automorphism groups of the original graph and the quotient graph.
Any $\alpha \in \aut(G)$ must take 
co-twins to 
co-twins, so $\alpha$ induces a permutation $\alpha^*$  of the collapsed vertices, given by 
$\alpha^*([u,v]) = [\alpha(u), \alpha(v)]$. Since any permutation of the vertices of a complete graph is a graph automorphism, this defines a group homomorphism $\kappa: \aut(G) \to \aut(K_k) = S_k$.

From subsection~\ref{subsec:TwinQuoGraph}, the subgroup generated by all twin transpositions consists of automorphisms in which sets of mutually twin vertices are permuted independently; this subgroup is the kernel of the projection of $\aut(G)$ to $\aut(\widetilde G)$.
The next result shows that co-twin pairs cannot be transposed independently.

\begin{lemma}\label{lem:kerKappa}
Let $\beta:V(G) \to V(G)$ be the function that transposes the vertices of every 
co-twin pair. Then the kernel of $\kappa:\aut(G) \to S_k$ is $\langle\beta\rangle$.
\end{lemma}

\begin{proof}
It is easy to verify that $\beta \in \aut(G).$ If $\alpha \in \ker(\kappa)$, then for each 
co-twin pair $\{u,v\}$, either $\alpha$ fixes both $u$ and $v$ or $\alpha$ transposes $u$ and $v$. In the latter case, $\alpha(N[u]) = N[v]$  and $\alpha(N[v]) = N[u]$. Since $N[u]$ and $N[v]$ are disjoint,  $\alpha$ cannot fix any vertex.  Hence if $\alpha$ transposes vertices in one 
co-twin pair, it must transpose the vertices of every 
co-twin pair and so $\alpha = \beta$.
\end{proof}

\begin{prop}\label{prop:Surjective}
Let $G$ be a graph of order $2k$ that is vertex-transitive and twin-free with 
co-twins. Then $\kappa: \aut(G) \to S_k$ is surjective iff $G$ is triangle-free.
\end{prop}

\begin{proof}

First assume $G$ is triangle-free.  Let $\sigma \in S_k$; we must define $\alpha \in \aut(G)$ such that $\kappa(\alpha) = \alpha^* = \sigma$.  
Let $[u,v]$ be a collapsed 
co-twin pair.
If $\sigma$ fixes $[u,v]$, then let $\alpha(u) = u$ and $\alpha(v) = v$.
If $\sigma([u,v]) = [x,y]$, where $[x,y]$ is a different collapsed 
co-twin pair, then by Lemma~\ref{lem:K2UK2}, the induced subgraph $G[u,v,x,y]$ is $K_2 \cup K_2$.
Without loss of generality,  
assume $u,x$ are adjacent and $v,y$ are adjacent, as in Figure~\ref{fig:K2UK2}. In this case, let $\alpha(u) = y$ and $\alpha(v) = x$; that is, if $\sigma$ moves $[u,v]$ to $[x,y]$, $\alpha$ takes each vertex in $\{u, v\}$ to the nonadjacent vertex in $\{x, y\}$.

Certainly $\alpha$ is a bijection on $V(G)$. 
To show $\alpha \in \aut(G)$, let $u, w \in V(G)$.
If $u$ and $w$ are adjacent, then they cannot be (nonadjacent) co-twins. Let $[u,v]$ and $[w,z]$ be the distinct 
co-twin pairs containing these vertices.
By Lemma~\ref{lem:K2UK2}, $v$ and $z$ are adjacent.

\medskip

\noindent {\bf Case 1.} If $\sigma$ fixes both $[u,v]$ and $[w,z]$, then $\alpha$ fixes both $u$ and $w$ and thus $\alpha(u)$ and $\alpha(w)$ are adjacent.

\medskip

\noindent {\bf Case 2.} Assume $\sigma([u,v]) = [x,y]$ but $\sigma$ fixes $[w,z]$. Assuming $u$ is adjacent to $x$ and $v$ is adjacent to $y$, by definition $\alpha(u) = y$ and $\alpha(w) = w$.
Since $G$ is triangle-free, $w$ cannot be adjacent to $x$, and so by Lemma~\ref{lem:K2UK2}, $w = \alpha(w)$ is adjacent to $y=\alpha(u)$.

\medskip 

\noindent {\bf Case 3.} Assume $\sigma([u,v]) = [x,y]$ and $\sigma([w,z]) = [r,s]$, where $u$ is adjacent to $x$ and $w$ is adjacent to $r$. 
Because $\sigma$ is injective, $[x,y] \neq [r,s]$. 
Then by definition $\alpha(u) = y$ and $\alpha(w) = s$.
As in the previous case, since $G$ is triangle-free and $u$ is adjacent to $w$,  $w$ is adjacent to $y$.
Since  $w$, $y$ and $r$ cannot induce a triangle in $G$, 
$y = \alpha(u)$ is adjacent to $s = \alpha(w)$.

Since $\alpha$ is an adjacency-preserving bijection on the vertices of a finite graph, $\alpha$ is an automorphism.
Clearly, $\kappa(\alpha) = \sigma$ and so $\kappa$ is surjective.

\medskip
Next assume that $h, j,\ell \in V(G)$ induce a triangle in $G$. Let $i,k,m$ be their respective 
co-twins; by Lemma~\ref{lem:K2UK2}, they also induce a triangle in $G$. Let $\tau \in S_k$ be the permutation that interchanges collapsed $[h,i]$ and $[j,k]$ but leaves all other vertices fixed. Assume $\kappa(\gamma) = \tau$ for some $\gamma \in \aut(G)$. 

By composing $\gamma$ with $\beta$ if necessary, we can assume that $\gamma$ fixes both $\ell$ and $m$. Since $\gamma$ respects adjacency, $\gamma(h) = j$ and $\gamma(i) = k$. Since $G$ has no adjacent twins, there is some vertex $p$ that is adjacent to $h$ but not to $j$. Let $q$ be the 
co-twin of $p$.

Since $\tau$ fixes $[p,q]$, $\gamma$ either fixes or interchanges $p$ and $q$. Since $p$ and $h$ are adjacent in $G$, $\gamma(p)$ and $\gamma(h) = j$ must be adjacent, so $\gamma(p) = q.$ By Lemma~\ref{lem:K2UK2}, $\{\ell, m\}$ and $\{p,q\}$ induce $K_2 \cup K_2$. Thus $\gamma$ cannot fix the vertices in the first 
co-twin pair and interchange them in the second.
\end{proof}

\begin{cor}\label{cor:AutCoTwins}
    Let $G$ be a graph of order $2k$ that is vertex-transitive, twin-free, with 
    co-twins. If $G$ is triangle-free, then $\aut(G) = S_k \times S_2$. 
\end{cor}

\begin{proof} 
By Lemma~\ref{lem:kerKappa}, Proposition~\ref{prop:Surjective}  and the First Isomorphism Theorem, $\aut(G)/\langle \beta \rangle \cong S_k$. Since $\langle \beta \rangle \cong \mathbb Z_2$, $|\aut(G)| = 2\cdot |S_k| = 2(k!)$.

Let $H$ be the set of automorphisms of $G$ used in the proof of Proposition~\ref{prop:Surjective}. More precisely, $\alpha \in  H$ iff whenever $\alpha$ fixes 
    co-twin pair $\{u,v\}$ set-wise, $\alpha$ fixes $\{u,v\}$ point-wise, and whenever $\alpha(\{u,v\}) = \{x,y\}$ for some $\{x,y\} \neq \{u,v\}$, then $\alpha$ takes each vertex in $\{u,v\}$ to the nonadjacent vertex in $\{x,y\}$.

Clearly $H$ contains the identity and is closed under inverses. Let $\alpha, \beta \in H$. To verify that $\beta \alpha \in H$, it suffices to consider the case 
 $\alpha(\{u,v\}) = \{x,y\}$ and $\beta(\{x,y\}) = \{w,z\}$ for $\{x,y\} \neq \{u,v\}$ and $\{w,z\} \neq \{x,y\}$. 
Assuming $u$ is adjacent to $x$ and $x$ is adjacent to $w$,  $\alpha(u) = y$ and $\beta(y) = w.$ 
Because $G$ is triangle-free,  $u$ is not adjacent $w$ and so $\beta \alpha$ takes each vertex in $\{u, v\}$ to the nonadjacent vertex in $\{w, x\}$.
Thus $H$ is a subgroup of $\aut(G)$ and  $H \cap \langle b \rangle$ is trivial. Moreover, the proof of Proposition~\ref{prop:Surjective} establishes a one-to-one correspondence between the elements of $S_k$ and $H$, meaning $|H| = |S_k| = k!$. Hence $H$ is a subgroup of $\aut(G)$ of index $2$ and so is normal. As a kernel, $\langle b \rangle$ is also normal. Hence $\aut(G)$ is the direct product of these two subgroups.
\end{proof}

\begin{prop}\label{prop:OnlyCrowns}
 A vertex-transitive, twin-free graph with 
 co-twins is triangle-free iff it is a crown graph $\overline{K_k \times K_2}$  for some $k \ge 3$.    
\end{prop}

\begin{proof}
   As noted earlier,  $\overline{K_k \times K_2}$, $k \ge 3$ is vertex-transitive, twin-free, with 
   co-twins. As a subgraph of $K_{kk}$, it is bipartite and hence triangle-free.

   Conversely, suppose $G$ is vertex-transitive, twin-free, with 
   co-twins that is triangle-free. By Lemma~\ref{lem:adjNonadjCoTwins}, $G$ has even order $2k$. All vertex-transitive graphs of orders $2$ and $4$ have twins, so $k \ge 3$. Let $\{u,v\}$ be a 
   co-twin pair and let $N(u) = \{x_1, x_2, \dots, x_{k-1}\}$. 
   Because $G$ is  triangle-free, no two vertices in $N(u)$ are adjacent. Since $N[u] \cap N[v] = \emptyset$, no $x_j$ is adjacent to $v$. Thus each $x_j$ must have $k-2$ neighbors in $N(v)$. If we let $y_j$ denote the 
   co-twin of $x_j$, then $y_j \in N(v)$ by Lemma~\ref{lem:K2UK2}. These considerations imply that $G$ must be as shown in  Figure~\ref{fig:OnlyCrowns}, which is clearly is $\overline{K_k \times K_2}.$   
\end{proof}

\begin{figure}[h]
    \centering
    \includegraphics[width= 0.3\textwidth, center]{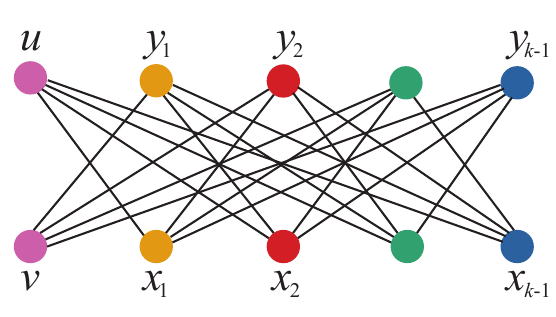}
    \caption{Vertex-transitive, twin-free, triangle-free with 
    co-twins.}
    \label{fig:OnlyCrowns}
\end{figure}

Since the automorphism group, determining number and distinguishing number of a graph are the same for a graph and its complement, we can combine results from Harary \cite{H1958}, Boutin \cite{B2009} and Imrich, Jerebic and Klavžar \cite{IJK2008} on direct products of complete graphs to obtain the following. 

\begin{prop}
For $k \ge 3$, $\aut(\overline{K_k \times K_2})  = S_k \times S_2$, $\det(\overline{K_k \times K_2}) = k-1$ and $\dist(\overline{K_k \times K_2})=d$ where $d \in \mathbb N$ satisfies  $(d-1)^2 \le k \le d^2 -1$.
\end{prop}

Now suppose that $G$ has order $n=2k$ is vertex-transitive, twin-free, with 
co-twins, but has triangles. 
We do not have a result as definitive as Corollary~\ref{cor:AutCoTwins}, but we can nonetheless draw conclusions about the automorphsim group and symmetry parameters of $G$.

\begin{prop}\label{prop:StabAut}
   Let $G$  be a graph of order $n=2k$ that is vertex-transitive, twin-free, with 
   co-twins. 
   Let $u \in V(G)$.
   Let $u\in V(G)$; let $H_u$ be the subgraph of $G$ induced by $N(u)$ and $\stab(u)$ be the stabilizer of $u$ in $\aut(G)$. Then
   \begin{enumerate}
   \item $\stab(u) \cong \aut(H_u)$,
   \item $|\aut(G)| = |\aut(H_u)| \cdot n$,
   \item $\det(G) = 1 + \det(H_u)$,
   \item $\dist(G) \le \dist(H_u)$.
   \end{enumerate}
\end{prop}

\begin{proof} All parts of the theorem clearly hold if $\aut(H_u)$ is trivial, so we assume $|\aut(H_u)| > 1$.
    Let $\pi:\stab(u) \to \aut(H_u)$ be the canonical restriction map, which is clearly a group homomorphism. If $\alpha \in \ker(\pi)$, then $\alpha$ fixes every vertex in $N[u]$. Because graph automorphisms preserve 
    co-twin pairs, $\alpha$ must also fix the 
    co-twin of every vertex in $N[u]$. By Lemma~\ref{lem:K2UK2}, no 
    co-twin pair can be entirely inside $N[u]$. Thus $\alpha$ is the identity automorphism.

    To show that $\pi$ is surjective, let $\gamma \in \aut(H_u)$. We want to extend $\gamma$ to a bijection  $\sigma$ on $V(G)$. We start by setting $\sigma(x) = \gamma(x)$ for all $x \in N(u)$.
    If $v$ is the  co-twin of $u$, then we set $\sigma(u) = u$ and $\sigma(v) = v$.  
    By Lemma~\ref{lem:K2UK2}, the 
    co-twin of each $w \in N[v]$ is a vertex in $N[u]$, call it $z$. We let $\sigma(w)$ be the 
    co-twin of $\gamma(z)$. 
    In other words, the action of $\sigma$ on the vertices of $N[v]$ `mirrors' the action of $\gamma$  on the vertices of $N[u]$. It is straightforward to verify that $\sigma \in \stab(u) \le \aut(G)$ and that $\pi(\sigma) = \gamma$.  

    Since $G$ is vertex-transitive, the orbit of $u$ is $V(G)$. The second statement of the proposition follows from the orbit-stabilizer theorem (see \cite{GR2001}).

    Because $G$ is nontrivial and vertex-transitive, $\det(G) >1.$ If $u$ is fixed, then to eliminate any nontrivial automorphism that fixes $u$, we must also fix vertices in a determining set of $H_u$. 

    Finally, we can construct a distinguishing coloring of $G$ by first using a palette of $\dist(H_u)$ colors to create a distinguishing color on $H_u$. Next, color the co-twin of each $v \in N(u)$ the same color as $v$. Finally, color $u$ and its co-twin with two different colors from the palette used for $H_u$. Let $\alpha\in \aut(G)$ be an automorphism that respects this coloring. Because $\alpha$ preserve co-twin pairs and only one co-twin pair has been assigned different colors, $\alpha(u) = u$.
\end{proof}

\begin{ex} 
The icosahedral graph is 
vertex-transitive, twin-free, with  co-twins and  
has triangles. Because $12 = 2\cdot 6$ and $6$ is even, it is not a circulant graph by Corollary~\ref{cor:OrderTwinFreeCoTwins}. 
For any vertex $u$, $H_u = C_5$, so $\aut(H_u) = D_5$, $\det(H_u) = 2 $ and $\dist(H_u) = 3$. Proposition~\ref{prop:StabAut} confirms the known facts that its automorphism group has order $12 \cdot 10 = 120$ and that its determining number and distinguishing number are both $3$.  
\end{ex}

\begin{ex} We apply Proposition~\ref{prop:StabAut} to the circulant graphs  $C_{14}(\pm 1, \pm 2, \pm 3)$ and $C_{18}(\pm 2, \pm 3, \pm 4, \pm 8)$, with $u = 0$.  
    The corresponding induced subgraphs $H_0$ are shown in Figure~\ref{fig:H0s}. In the first case, 
    $\aut(H_0) = S_2$, 
    which implies that $\mathbb Z_{14}$ has index 2 in $\aut(C_{14}(\pm 1, \pm 2, \pm 3))$ and thus is normal. As discussed at the end of Section \ref{sec:Circulants}, this implies
$\aut(C_{14}(\pm 1, \pm 2, \pm 3)) = \mathbb Z_{14} \rtimes S_2 = D_{14}$. 
Since $\det(H_0) = 1$, $\dist(H_0)=2$ and $C_{14}(\pm 1, \pm 2, \pm 3)$ has nontrivial automorphisms,
\[
\det(C_{14}(\pm 1, \pm 2, \pm 3)) = \dist(C_{14}(\pm 1, \pm 2, \pm 3)) =2. 
\]
    In the second case, $H_0 = 2K_1 \cup K_{33}$, so $\aut(H_0) = S_2 \times \big ((S_3 \times S_3) \rtimes S_2\big)$, which has order $144$. Thus 
    $|\aut(C_{18}(\pm 2, \pm 3, \pm 4, \pm 8))| = 144 \cdot 18 = 2592.$
    Since $\det(H_0) = \det(K_{33}) + 1 =5$ and $\dist(H_0) =\dist(K_{33}) = 4$,     
    \[
   \det(C_{18}(\pm 2, \pm 3, \pm 4, \pm 8)) = 6 \text{ and } \dist(C_{18}(\pm 2, \pm 3, \pm 4, \pm 8)) \le 4. 
    \]
    \end{ex}

 \begin{figure}[h]
    \centering
    \includegraphics[width= 0.65\textwidth, center]{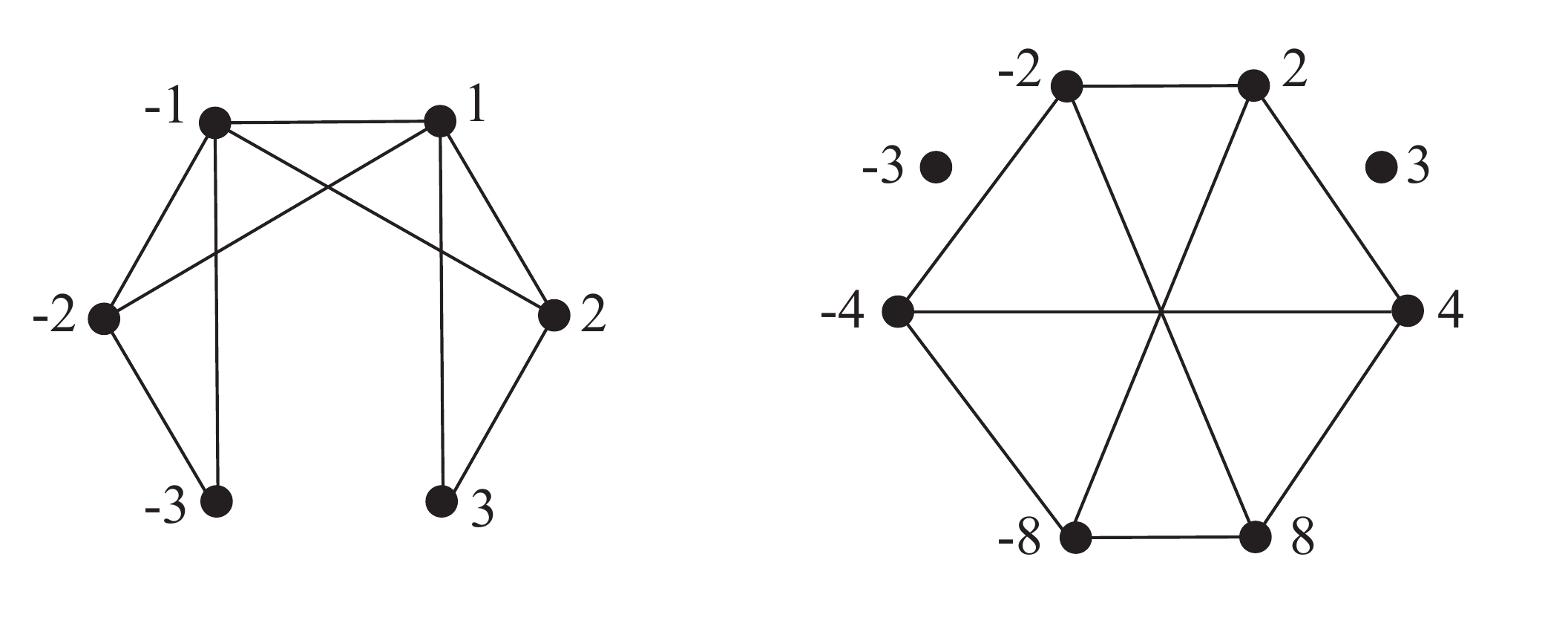}
    \caption{$H_0$ for $C_{14}(\pm 1, \pm 2, \pm 3)$ and $C_{18}(\pm 2, \pm 3, \pm 4, \pm 8)$.}
    \label{fig:H0s}
\end{figure}

\section{Future Work}

The results in this paper can be used to address the notoriously difficult problem of finding automorphism groups of circulant graphs by tracking the effects of the presence twins and co-twins. We have demonstrated that in some special cases, our results completely determine the automorphism group, determining number and distinguishing number. There is ample scope for future work in finding the automorphsim groups and symmetry parameters of circulant graphs which are both twin-free and co-twin-free.

\begin{appendix}
 \section{Proof of Proposition~\ref{prop:Twins3Gen} when $k < n/2$}

 We restate the proposition. 
 \begin{prop}\label{prop:Twins3Gen}
Let $0<i<j<k\le n/2$, with $\gcd(i,j,k,n) = 1$.
Then $C_n(i, j, k)$  has twins only 
shown in Table~\ref{table:ThreeGenTwins}. The twin classes of the non-complete graphs are the cosets of $\langle w \rangle$.

\begin{table}[h]
    \center
    \begin{tabular}{|c|c|c|} \hline
     & Nonadjacent twins & Adjacent twins\\
    \hline
    $k = n/2$ & $C_{10}(1,3,5) = K_{5,5}: w = 2$ & $C_6(1,2,3) = K_6$\\
    &&$i+j=k = \tfrac{n}{2}: w = \tfrac{n}{2}$ \\ \hline 
     & $C_{12}(1, 3, 5) = K_{6,6}: w = 2$ & $C_7(1,2,3) = K_7$ \\
    $k < n/2$ &   $i+j=\tfrac{n}{3}, 2i+j=k: w = \tfrac{n}{3}$ & \\
    &  $i+k = 2j = \tfrac{n}{2}: w = \tfrac{n}{2}$ &\\ \hline  
    \end{tabular}
    \caption{Connected three-generator circulant graphs with twins.\label{table:ThreeGenTwins}}    
\end{table}

\end{prop}

 \begin{proof} The proof in the case $k = n/2$ appears in the body of the paper.
   So assume $k < n/2$, which implies $|A|=6$. If $C_n(i,j,k)$ has nonadjacent twins, there are three cases to consider.
If $A$ is a single nontrivial coset of some $w \in \mathbb Z_n$ of order $6$, then matching the integer sequences
 \[0<i< i+\tfrac{n}{6}< i+\tfrac{2n}{6}< i+\tfrac{3n}{6}< i+\tfrac{4n}{6} < i+ \tfrac{5n}{6}<n \]
and 
\[0<i<j<k < n-k <n-j < n-i <n\]
gives
$n = 12i$, $j = 3i$ and $k=5i$; connectivity forces $i=1$. 
If $A$ is a union of two nontrivial cosets of some $w \in \mathbb Z_n$ of order $3$, then 
\[A = \{i, i+\tfrac{n}{3}, i+\tfrac{2n}{3}\} \cup \{\ell, \ell+\tfrac{n}{3}, \ell+\tfrac{2n}{3}\}\] for some $\ell \in \mathbb Z$ satisfying $i < \ell < n/3$. This implies that as integers, \[0<i<\ell < i+w < \ell + w < i+2w <\ell + 2w < n.\]  This forces $j = \ell$ and $k = i+w$. Since $A = -A$, $i + j = w$ and $2i+j = k$.
If $A$ is a union of three nontrivial cosets of $w \in \mathbb Z_n$ of order $2$, then 
\[A = \{i, i+\tfrac{n}{2}\}\cup \{j, j+\tfrac{n}{2}\} \cup \{k, k+\tfrac{n}{2}\}.\] As integers,
    \[0<i<j<k<i+\tfrac{n}{2}<j+\tfrac{n}{2}<k+\tfrac{n}{2} <n.\] Because $A$ is inverse-closed, $i+k = 2j = n/2$.

By Corollary~\ref{cor:WhenAdjTwins}, $C_n(i,j,k)$ has adjacent twins iff $A \cup\{0\}$ is an inverse-closed union of cosets. The only possibility is that $A\cup\{0\}$ is the trivial coset of an element $w$ of order $7$. The only such circulant graph that is connected is $C_7(1,2,3) = K_7$. 
\end{proof}
\end{appendix}

\bibliographystyle{abbrv}
\bibliography{bibliograph}
\end{document}